\documentclass[a4paper, 11 pt, twoside]{amsart}
\usepackage{srollens-en}[2014/04/11]
\usepackage[left = 3.5cm, right = 3.5cm, headsep = 6mm,
footskip = 10mm, top = 35mm, bottom = 35mm, footnotesep=5mm, headheight =
2cm]{geometry}

\usepackage{caption}
\usepackage{subcaption}
\DeclareCaptionSubType*[Alph]{figure}
\usepackage{booktabs, multirow}
\usepackage{enumitem}

\usepackage[usenames, dvipsnames]{xcolor}
\usepackage{tikz, tikz-cd}
\usetikzlibrary{through}
\usetikzlibrary{decorations.pathreplacing,decorations.markings}
\tikzset{
  on each segment/.style={
    decorate,
    decoration={
      show path construction,
      moveto code={},
      lineto code={
        \path [#1]
        (\tikzinputsegmentfirst) -- (\tikzinputsegmentlast);
      },
      curveto code={
        \path [#1] (\tikzinputsegmentfirst)
        .. controls
        (\tikzinputsegmentsupporta) and (\tikzinputsegmentsupportb)
        ..
        (\tikzinputsegmentlast);
      },
      closepath code={
        \path [#1]
        (\tikzinputsegmentfirst) -- (\tikzinputsegmentlast);
      },
    },
  },
  mid arrow/.style={postaction={decorate,decoration={
        markings,
        mark=at position .5 with {\arrow[#1]{stealth}}
      }}},
}

\usepackage{pdflscape}

\usepackage[ocgcolorlinks, unicode,bookmarks, pdftex, backref = page]{hyperref}
\hypersetup{colorlinks=true,citecolor=NavyBlue,linkcolor=NavyBlue,urlcolor=Orange, pdfpagemode=UseNone, breaklinks=true}

\numberwithin{equation}{section}

\renewcommand{\tilde}{\widetilde}

\newcommand{\pp}{\mathbb P}

\newcommand{\OO}{\mathcal O}
\newcommand{\Diff}{\mathrm{Diff}}

\newcommand{\I}{\mathrm{i}}

\title[Computing invariants of semi-log-canonical surfaces]{Computing 
 invariants of semi-log-canonical surfaces}

\author{Marco Franciosi}
\address{Marco Franciosi\\Dipartimento di Matematica\\Universit\`a di Pisa \\Largo B. Pontecorvo 5\\I-56127  Pisa\\Italy}
\email{franciosi@dm.unipi.it}

\author{Rita Pardini}
\address{Rita Pardini\\Dipartimento di Matematica\\Universit\`a di Pisa \\Largo B. Pontecorvo 5\\I-56127  Pisa\\Italy}
\email{pardini@dm.unipi.it}

\author{S\"onke Rollenske}
\address{S\"onke Rollenske\\Fakult\"at f\"ur Mathematik\\Universt\"at Bielefeld\\Universit\"atsstr. 25\\33615 Bielefeld\\Germany}
\email{rollenske@math.uni-bielefeld.de}

\subjclass[2010]{14J10, 14J29, 14F35}
\keywords{stable surfaces, irregular surfaces, moduli of surfaces of general type}

\begin{document}
\begin{abstract}
We describe some methods to compute fundamental groups, (co)homo\-logy, and irregularity of semi-log-canonical surfaces.

As an application, we show that there are exactly two irregular Gorenstein stable surfaces with $K^2=1$, both of which have $\chi(X) = 0$ and  $\Pic^0(X)=\IC^*$ but different homotopy type. 
\end{abstract}

\maketitle

\setcounter{tocdepth}{1}
\tableofcontents

\section{Introduction}
Whenever one finds a new way to construct interesting varieties, one is subsequently presented with the task of computing algebraic and topological invariants. We address in this paper the case of  surfaces with semi-log-canonical (slc) singularities. 

Semi-log-canonical   surfaces with  ample canonical divisor are  called stable  and their moduli space  is a natural compactification of  the Gieseker moduli space of canonical models of surfaces of general type (see \cite{kollar12, kollarModuli}). Indeed, this was one of the motivations for the introduction of this class of singularities by Koll\'ar and Shepherd-Barron in \cite{ksb88}. 

Here we start by    giving   a general method  to compute (co)homology  and fundamental group of a variety (or complex space) from a birational modification (cf.  Section \ref{section: computations}); in the applications this can be, for instance,   the normalisation or a partial resolution. 
Then we  explain how to compute   the irregularity $q=h^1(\OO_X)$ and also the automorphism group of a non-normal slc surface. 

Our motivation for these results comes from our work in progress on  Gorenstein stable surfaces with $K^2=1$ begun in \cite{fpr14}.   More precisely, since  minimal surfaces of general type with $K^2=1$  are known to be regular, we wished to  decide  whether  the same is true in the  stable  case.
In the smooth case regularity is  proven via  a covering trick:  if $q>0$ there exist \'etale coverings of arbitrary degree and these would violate the Noether-inequality (\cite[Lem.~14]{bombieri73}).

For stable surfaces, even Gorenstein, this argument does not work because  the slope of the Noether line is 1 \cite{liu-rollenske13}. Still, the situation is remarkably similar:
\begin{custom}[Theorem A]
There are exactly two irregular Gorenstein stable surfaces with $K^2=1$. They  both  have $\chi(X)=0$ and $q(X)=1$, the same  normalisation ($=\IP^2$), the same integral homology and Picard group, but they have different homotopy type.

Each one  of these  surfaces corresponds to   a connected component of the moduli space of (not necessarily Gorenstein) stable surfaces with $K^2=1$. 
\end{custom}
Since smooth  surfaces of general type have $\chi(X)>0$, the irregular surfaces we find  cannot be smoothable. A natural question raised by Theorem A is whether there are non-Gorenstein stable surfaces with $K^2=1$ and $\chi\le 0$. 

Our proof consists of two steps: first  in Section \ref{section: classical} we prove by ``classical'' methods (restriction to a canonical curve) that if $\chi(X)>0$ and $K^2=1$ then $q(X)=0$. 
By the results of \cite{fpr14},  this implies that an irregular surface $X$ with $K^2=1$  has  $\chi(X)=0$ and  is a projective plane glued to itself along four lines. 
Secondly,  in Section~\ref{section: 4 lines} we classify  Gorenstein stable surfaces  constructed from the projective plane glued to itself along four lines, showing in particular that there are exactly two with $\chi(X)=0$.  The topological and algebraic invariants  of these surfaces are  computed in the  last section \ref{section: chi=0}, using the methods previously devised. 

In a forthcoming work we will study the Gorenstein stable surfaces with $K^2 =1$ and $\chi(X)>0$ (hence  $q(X)=0$) and their moduli spaces.

\subsection*{Acknowledgements}
 The first author  is a member of GNSAGA of INDAM. The third author is grateful for support of the DFG through the Emmy Noether program and SFB 701. The collaboration  benefited immensely from a visit of the third author in Pisa supported by GNSAGA of INDAM. This project was partially supported by PRIN 2010 ``Geometria delle Variet\`a Algebriche'' of italian MIUR. 
 
We are indepted to Stefan Bauer, Kai-Uwe Bux, Michael L\"onne, Hanno von Bodecker for guidance around the pitfalls of algebraic topology. Kai-Uwe Bux also explained to us how to prove that the fundamental  groups computed in Proposition~\ref{prop; topological invariants} are not isomorphic. 
The third author is grateful to Wenfei Liu for many discussions on stable surfaces and to Filippo Viviani for some helpful email communication.

\section{Semi-log-canonical  surfaces}\label{sec:slc_surf}
In this section we recall briefly the  facts about semi-log-canonical   surfaces that we use  later. 
Let $X$ be a {\em demi-normal } surface, that is,  $X$ satisfies Serre's condition $S_2$ and  at each point of codimension one either it  is  regular or it has an ordinary double point.
We denote by  $\pi\colon \bar X \to X$ the normalisation of $X$.
Observe that  $X$ is not assumed irreducible; in particular, $\bar X$ is possibly disconnected.
The conductor ideal
$ \shom_{\ko_X}(\pi_*\ko_{\bar X}, \ko_X)$
is an ideal sheaf in both $\ko_X$ and $\ko_{\bar X} $ and as such defines subschemes
$D\subset X \text{ and } \bar D\subset \bar X,$
both reduced and pure of codimension 1; we often refer to $D$ as the {\em non-normal locus} of $X$.

\begin{defin}\label{defin: slc}
The  demi-normal surface $X$ is said to have \emph{semi-log-canonical (slc)}  singularities if it satisfies the following conditions: 
\begin{enumerate}
 \item The canonical divisor $K_X$ is $\IQ$-Cartier.
\item The pair $(\bar X, \bar D)$ has log-canonical (lc) singularities (cf. \cite[\S 4.1]{Kollar-Mori}). 
\end{enumerate}
It  is called a {\em stable}  surface 
 if in addition $K_X$ is ample. In that case we define the {\em geometric genus} of $X$ to be $ p_g(X) = h^0(X, K_X) = h^2(X, \ko_X)$ and the {\em irregularity} as $q(X) = h^1(X, K_X) = h^1(X, \OO_X)$ (cf. \cite[Lem.~3.3]{liu-rollenske14} for the last equality). 
A Gorenstein stable surface is a stable surface such that $K_X$ is a Cartier divisor.
\end{defin}
\subsection{Koll\'ar's gluing principle}\label{ssec: kollar}

Let $X$ be a demi-normal surface. Since $X$ has at most double points in codimension one, the map $\pi\colon \bar D \to D$ on the conductor divisors is generically a double cover and thus  induces a rational involution on $\bar D$. Normalising the conductor loci we get an honest involution $\tau\colon \bar D^\nu\to \bar D^\nu$ such that $D^\nu = \bar D^\nu/\tau$, where $\bar D^{\nu}$, $D^{\nu}$ are the normalizations of $\bar D$, resp. $D$.

\begin{theo}[{\cite[Thm.~5.13]{KollarSMMP}}]\label{thm: triple}
Associating to a stable surface $X$ the triple $(\bar X, \bar D, \tau\colon \bar D^\nu\to \bar D^\nu)$ induces a one-to-one correspondence
 \[
  \left\{ \text{\begin{minipage}{.12\textwidth}
 \begin{center}
         stable  surfaces  
 \end{center}
         \end{minipage}}
 \right\} \leftrightarrow
 \left\{ (\bar X, \bar D, \tau)\left|\,\text{\begin{minipage}{.37\textwidth}
   $(\bar X, \bar D)$ log-canonical pair with 
  $K_{\bar X}+\bar D$ ample, \\
   $\tau\colon \bar D^\nu\to \bar D^\nu$  involution s.th.\
    $\Diff_{\bar D^\nu}(0)$ is $\tau$-invariant.
            \end{minipage}}\right.
 \right\}.
 \]
 \textbf{\upshape Addendum:} In the above correspondence the surface $X$ is Gorenstein if and only if  $K_{\bar X}+\bar D$ is Cartier and $\tau$ induces a fixed-point free involution on the preimages of the nodes of $\bar D$.\end{theo}

For the definition of the different see for example \cite[5.11]{KollarSMMP}. The addendum is proven in \cite[\S 3.1]{fpr14}.

An important consequence, which allows to understand the geometry of stable  surfaces from the normalisation, is that 
\begin{equation}\label{diagr: pushout}
\begin{tikzcd}
    \bar X \dar{\pi}\rar[hookleftarrow]{\bar\iota} & \bar D\dar{\pi} & \bar D^\nu \lar[swap]{\bar\nu}\dar{/\tau}
    \\
X\rar[hookleftarrow]{\iota} &D &D^\nu\lar[swap]{\nu}
    \end{tikzcd}
\end{equation}
is a pushout diagram.

\begin{rem} \label{rem:pushout}
Consider diagram \eqref{diagr: pushout} for   any  demi-normal surface $X$, not necessarily stable.  By  \cite[Prop.~5.3]{KollarSMMP} and  proof thereof,  $X$   has  the following universal property: given any finite surjective  morphism $f\colon \bar  X\to Y$ that induces a $\tau$-invariant map $\bar D^{\nu}\to Y$, there is a unique morphism $g\colon X\to Y$ such that $f=g\circ \pi$.

\end{rem}

\section{Computing invariants}\label{section: computations}
In this section we illustrate some methods for the computation of invariants  of  non-normal surfaces. We apply these results  to stable Gorenstein surfaces with $K^2=1$ in Section \ref{sec:k2=1}.

\subsection{Topology}
First we compute the fundamental group and (co)homology of a non-normal surface in terms of its normalisation. 

The proof of the following result is a synthesis of some conversations with  Stefan Bauer, Kai-Uwe Bux, Michael L\"onne, and Hanno von Bodecker.
It is probably well known to experts.
\begin{prop}\label{prop: homotopy pushout}
 Let $\pi\colon \bar X\to X$ be a  holomorphic map of compact complex analytic spaces. Assume $A$ is a closed  analytic subspace of $X$ such that,  if we set   $E = \inverse\pi A$,  the map $\pi\colon \bar X \setminus E\to X\setminus A$ is an isomorphism. 

Let $M$ be the double mapping cylinder of $\pi\colon E \to A$ and the inclusion $\iota\colon E\to \bar X$, that is, $M = A\cup_\pi  E\times I \cup_\iota \bar X$ where we glue  $E\times\{0\}$ to $A$ via $\pi$ and $E\times \{1\}$ to $\bar X$ via $\iota$:
{
\small
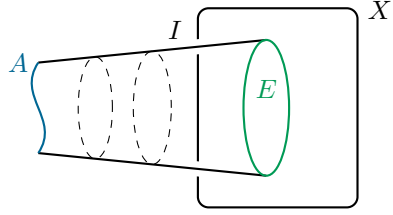
\[\begin{tikzpicture}[thick, scale = 0.6]
 \draw[rounded corners] (-2.5, -2.2) rectangle (1, 2.2) node [right]{$\bar X$};
\draw[line width =.2cm, white]  (-6,1) -- (-1,1.5) (-6,-1) -- (-1,-1.5);
\draw (-6,1) to node[pos=.6, above]{$I$} (-1,1.5) (-6,-1) to  (-1,-1.5);
\draw[MidnightBlue] (-6,1)  node[left]{$A$}  .. controls (-6.5,.25) and (-5.5,-.25) ..(-6,-1);
 \draw[thin, dashed] (-3.5,0) circle [ x radius = .4166cm, y radius = 1.25cm];
 \draw[thin, dashed] (-4.75,0) circle [ x radius = .375cm, y radius = 1.125cm];
 \draw[ForestGreen] (-1,0) circle [ x radius = .5cm, y radius = 1.5cm] node[above]{$E$};
\end{tikzpicture}
\]
}
Then the natural map $M\to X$ is a homotopy equivalence.
\end{prop}
\begin{proof}
 By \cite{lojasiewicz64} (see also \cite[Thm.~I.8.8]{BHPV}) the subspace $E$ is a neighbourhood deformation retract in $\bar X$. This implies  \cite[Ch.\ 6,\  \S\ 4]{May}  that the inclusion $E \into \bar X$ is a cofibration and we conclude by \cite[Ch.\ 10,\  \S\ 7, Lemma on p.\ 78]{May}.
\end{proof}
We formulate the application of the above result only in the case we use, that is, for the normalisation of a demi-normal surface.
\begin{cor}\label{cor: mayer vietoris and fundamental group as amalgamated product}
 Let $X$ be a demi-normal surface and let $D$, $\bar D$, and $\bar X$  be as in  \eqref{diagr: pushout}. Then:
\begin{enumerate}
 \item There is a   Mayer-Vietoris  exact sequence for homology
\[ \begin{tikzcd}
{}\rar& H_i(\bar D, \IZ) \rar{(\pi_*, \bar \iota_*)} & H_i(D, \IZ)\oplus H_i(\bar X, \IZ)\rar{\iota_*-\pi_*} & H_i(X, \IZ)\rar&
\end{tikzcd}\]
\item  Suppose  in addition that $\bar D$ is connected.
Then 
\[\pi_1(X) \isom \pi_1(D)\star_{\pi_1(\bar D)}\pi_1(\bar X),\]
where $\star$ denotes  amalgamated product of groups.
\item Suppose that $\bar D$ is connected and $\pi_1(D)$ and $\pi_1(\bar X)$ are abelian. Then
\[\pi_1(X) \isom H_1(D, \IZ)\star_{H_1(\bar D, \IZ)}H_1(\bar X, \IZ).\]
\end{enumerate}
\end{cor}
\begin{proof}
We apply Proposition \ref{prop: homotopy pushout} to $\pi \colon \bar X\to X$ with $E=\bar D$ and $A= D$. 
Choose a base point $x_0\in \bar D$ and take as a base point on  the homotopy model  $Z\cup_{\bar D}\bar X$ the point $(x_0, 1/2)$ on the mapping cylinder. The open set $\{(x,t)\in Z\mid t>1/4\}$ and the complement of $\{(x,t)\in Z\mid t\leq3/4\}$ cover $Z\cup_{\bar D}\bar X$. Applying Mayer-Vietoris, respectively  Seifert-van Kampen, to this decomposition gives the claimed result.
\end{proof}
\begin{rem}
 In the category of complex algebraic varieties the cohomology groups carry a mixed Hodge structure; the compatibility of the Mayer-Vietoris sequence with this additional structure is proven in \cite[ 5.37]{Peters-Steenbrink}.
\end{rem}

\subsection{Automorphisms}\label{ssec:aut}
Let $X$ be a demi-normal   surface and let $\sigma\in \Aut(X)$. Then $\sigma$ induces an automorphism $\bar\sigma$ of $\bar X$ such that $\bar\sigma(\bar D)=\bar D$; we denote by  $\bar\sigma^{\nu}$ the automorphism  of $\bar D^{\nu}$ induced by the restriction of $\bar \sigma$. Clearly $\tau$ and $\bar\sigma^{\nu}$ commute. We denote by $\Aut(\bar X, \bar D, \tau)$ the subgroup of $\Aut(\bar X)$ consisting of the automorphisms $\psi$ such that $\psi(\bar D)=\bar D$ and the induced automorphism $\psi^{\nu}$ of $\bar D^{\nu}$ commutes with $\tau$. The map  $\Aut(X)\to \Aut(\bar X, \bar D,\tau)$ defined by $\sigma\to \bar \sigma$ is clearly  an injective homomorphism.
In addition, one has:
\begin{lem}\label{lem:autXbar} Let $X$ be a demi-normal surface obtained from a triple  $(\bar X,\bar D, \tau)$.
Then $\Aut(X)\to \Aut(\bar X, \bar D,\tau)$ is an isomorphism.

\end{lem}
\begin{proof} It suffices to prove that the map is surjective, namely that given an automorphism  $\psi\in \Aut(\bar X, \bar D, \tau)$ there exists $\sigma \in \Aut(X)$ such that $\bar \sigma=\psi$.  Applying   Remark \ref{rem:pushout} to the finite surjective map  $f=\pi\circ \psi\colon \bar X\to X$  one obtains    $\sigma\colon X\to X$ such that $\bar \sigma=\psi$. By general nonsense $\sigma$ is indeed an automorphism.  \end{proof}

\subsection{Irregularity}\label{section: irregularity}
Let $X$ be an slc surface, not necessarily Gorenstein. While it is easy to compute $\chi(X)$ from the normalisation, the calculation of the irregularity (or geometric genus) is more subtle: the irregularity can either  drop or increase in the normalisation (see \cite[Sect.~5.3.1]{liu-rollenske13} for some examples.)

In this section, we give an algorithm to compute the irregularity in concrete examples, at least as long as  the normalisation has irregularity $q=0$; 
the reason for considering here  only slc surfaces is that in our computations we exploit the classification of slc surface singularities.

 We keep  the notations from Diagram \eqref{diagr: pushout}.
We define a sheaf $\kq$ via the following diagram:
\begin{equation}\label{eq: q1}
\begin{tikzcd}
 {} & 0\dar & 0\dar & 0\dar\\
0\rar& \ki_D \rar\dar[equal] & \ko_X \rar\dar& \ko_D\rar\dar& 0\\
0\rar& \pi_*\ko_{\bar X}(-\bar D) \rar & \pi_*\ko_{\bar X} \rar\dar&\pi_*\ko_{\bar D}\rar\dar& 0\\
&& \kq \rar[equal]\dar&\kq \dar\\
&&0&0
\end{tikzcd}
\end{equation}
Note that the two rows of the diagram give the formula
\begin{equation}
\label{eq: formula chi}\chi(X)=\chi(\bar X)-\chi(\bar D)+\chi(D).
\end{equation}
We wish   to use the second column to compute also the irregularity  $q(X)$, so we need to understand  $\kq$. First note that $\kq$ is a torsion-free sheaf on the curve $D$, because $\bar D$ is a curve without embedded points. We now compare the third column with the analogous sequence on the normalisations.
 Recall that $\pi\colon \bar D^\nu\to D^\nu$  is a double cover with involution $\tau$, so $\pi_*\ko_{\bar D^\nu}=\ko_{D^\nu}\oplus \inverse\kl$ where $\kl^{\tensor 2}$ is the line bundle associated with  the branch locus. Then there is a diagram with exact rows and columns
\begin{equation}\label{eq: q2}
\begin{tikzcd}
{}&0 \dar&0\dar&0\dar\\
0\rar& \ko_D \rar\dar& \nu_*\ko_{D^\nu} \rar\dar& \nu_*\ko_{D^\nu}/\ko_D\rar\dar& 0\\
0\rar& \pi_*\ko_{\bar D} \rar\dar& \pi_*\bar\nu_*\ko_{\bar D^\nu} \rar\dar& \pi_*\left(\bar \nu_*\ko_{\bar D^\nu}/\ko_{\bar D}\right)\rar\dar& 0\\
0\rar& \kq \rar\dar& \nu_*\inverse\kl \rar\dar& \nu_*\inverse\kl/\kq\rar\dar& 0\\
&0&0&0
\end{tikzcd}
\end{equation}
Indeed, the first and second row, respectively column, are easily seen to be exact. Let $\kk$ be the kernel of  $\kq\to\nu_*\inverse\kl$. By the snake Lemma it injects into $\nu_*\ko_{D^\nu}/\ko_D$ and thus is supported at finitely may points. On the other hand it is included in the torsion-free sheaf on $\kq$ and thus $\kk=0$ and also the third row, respectively column, is exact.

We need to introduce some further notation locally analytically at a point $p\in D$. If $p$ is semi-log-terminal  then $\inverse\pi(p)$ does not contain nodal points of $\bar D$ (cf. \cite[Thm.~4.23]{ksb88}) and thus $\kq\isom \nu_*\inverse\kl$ at $p$  by the exactness of the third column.

So assume that $p\in D$ is a degenerate cusp or a quotient thereof. For the details of the following discussion we refer to \cite[Sect.~4.1]{liu-rollenske14}. Let $q_1, \dots, q_{\mu}\in \bar D$ be the preimages of $p$. If $q_i$ is a nodal point of $\bar D$, then we denote its two preimages in $\bar D^\nu$ by $r_i$ and $s_i$.

If $p$ is a degenerate cusp then all the $q_i$ are nodal points of $\bar D$ and we can order the $q_i$ in such a way, that $\tau(s_i) = r_{i+1}$ where the index is computed modulo $\mu$. We illustrate this in Figure \ref{fig: local at deg cusp} for $\mu=3$. 

If $p$ is a quotient of a degenerate cusp, then one can order the $q_i$ in such a way that  $q_2, \dots , q_{\mu-1}$ are nodes and $\tau$ acts as above for $\mu=2,\dots, \mu-2$. The \emph{end-points} $q_1$ and $q_{\mu}$ can be of two types: either a nodal point of $\bar D$ which is the image of a  fixed point  of the involution $\tau$ or a dihedral point of $\bar X$ which is a smooth point of $\bar D$. If $q_1$ is a smooth point of $\bar D$, then its preimage in $\bar D^\nu$ is mapped  to $r_2$ by $\tau$, while if it is a node  with preimages $r_1, s_1$ then $s_1$ is mapped to $r_2$ by $\tau$ while $s_1$ is fixed by $\tau$. The situation for $q_{\mu}$ is similar.

\begin{figure}[t]
 \begin{tikzpicture}
[thick, every path/.style ={ fill=black, radius = 2.5pt},
every node/.style ={black},
L1/.style = {MidnightBlue},
L2/.style = {ForestGreen},
L3/.style = {RedOrange},
scale = .6]

\begin{scope}[ xshift = 0cm, yshift = 10cm]
\node at (-3,3) {$\bar D$};
\draw[L3] (0,0) ++(285:0.5)  ++(45: .25cm) -- ++(225:3cm);
\draw[L1] (0,0)++(285:0.5)++(-.25,0) -- ++(3,0);
\draw[L2] (0,0)++(165:0.5)++(0,-.25) -- ++(0,3);
\draw[L2] (45:0.5) ++(0,-.25) -- ++(0,3);
\draw[L3] (0,0)++(165:0.5)  ++(45: .25cm) -- ++(225:3cm);
\draw [L1] (45:0.5)++(-.25,0) -- ++(3,0);
\draw (0,0) ++(285:0.5) circle node[below right]{$q_1$};
\draw (45:0.5) circle  node[above right]{$q_2$};
\draw (0,0)++(165:0.5) circle  node[left]{$q_3$};
\draw[->] (0, -3) to  node[left] {$\pi$} ++(0,-1.5);
\end{scope}

\begin{scope}[xshift = 10cm, yshift = 10cm]
\node at (-3,3) {$\bar D^\nu$};

\draw[L3] (-0.5,-0.5)++(135:0.25)circle  node[above left]{$s_3$} ++(45: .25cm) -- ++(225:3cm);
\draw[L3] (-0.5,-0.5)++(135:-0.25) circle  node[below right]{$r_1$} ++(45: .25cm) -- ++(225:3cm);
\draw[L1] (1,-.25) circle  node[below right]{$s_1$} ++(-.25,0) -- ++(3,0);
\draw[L1] (1,.25) circle  node[above right]{$r_2$} ++(-.25,0)  -- ++(3,0);
\draw[L2] (.25,1) circle  node[above right]{$s_2$} ++(0,-.25) -- ++(0,3);
\draw[L2] (-.25,1) circle  node[above left]{$r_3$} ++(0,-.25) -- ++(0,3);
\draw[->] (0, -3) to  node[left] {$\pi$} ++(0,-1.5);
\draw[->] (-4, 0) to  node[above] {$\bar \nu$} ++(-1.5,0);

\draw (-0.5,-0.5)++(135:0.25)circle  node[above left]{$s_3$};
\draw (-0.5,-0.5)++(135:-0.25) circle  node[below right]{$r_1$};
\draw (1,-.25) circle  node[below right]{$s_1$};
\draw (1,.25) circle  node[above right]{$r_2$};
\draw (.25,1) circle  node[above right]{$s_2$};
\draw (-.25,1) circle  node[above left]{$r_3$};

\end{scope}

\begin{scope}
\node at (-3,3) {$D$};
\draw[L3] (0,0) circle  node[below right]{$p$} ++(45: .25cm) -- ++(225:3cm);
\draw[L2] (0,-.25) -- ++(0,3);
\draw[L1] (-.25,0) -- ++(3,0);
\draw (0,0) circle  node[below right]{$p$};
\end{scope}

\begin{scope}[xshift = 10cm]
\node at (-3,3) {$D^\nu$};
\draw[L3] (-0.5,-0.5) circle  node[below right]{$p_1$}++(45: .25cm) -- ++(225:3cm);
\draw[L1] (1,0) circle  node[above right]{$p_2$} ++(-.25,0) -- ++(3,0);
\draw[L2] (0,1) circle  node[above left]{$p_3$}++(0,-.25) -- ++(0,3);
\draw[->] (-4, 0) to  node[above] {$ \nu$} ++(-1.5,0);
\draw (-0.5,-0.5) circle  node[below right]{$p_1$};
\draw (1,0) circle  node[above right]{$p_2$};
\draw (0,1) circle  node[above left]{$p_3$};
\end{scope}
\end{tikzpicture}
\caption{Local notation at a degenerate cusp $p$ with $\mu=3$: $\bar\nu(r_i) = \bar\nu (s_i) = q_i$ and $ \tau(s_i) = r_{i+1}$.}\label{fig: local at deg cusp} 
\end{figure}
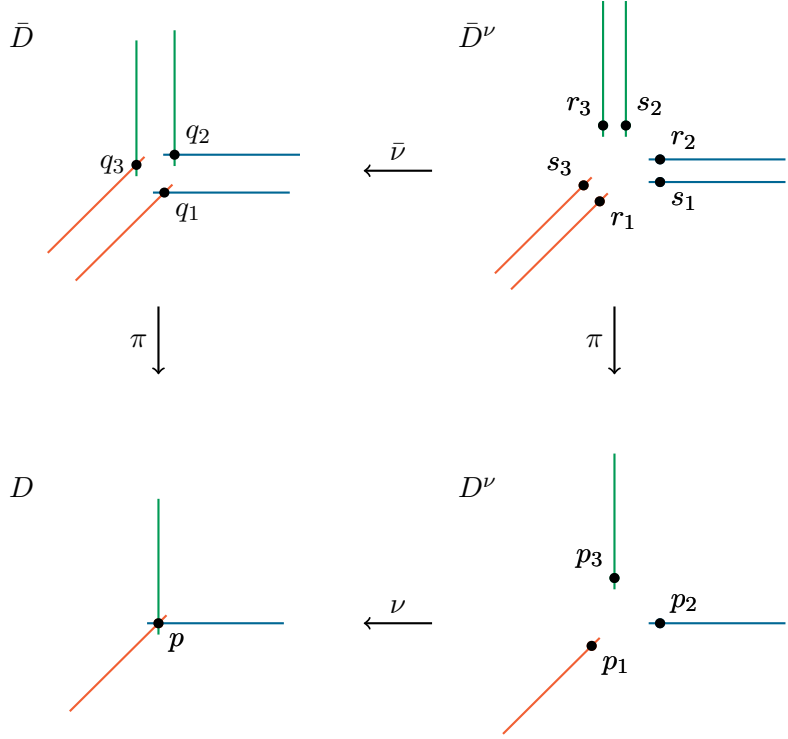

\begin{lem}\label{lem: L^-1/Q} With  the above notations, 
  let $p\in D$ be a point. Then 
\begin{enumerate}
 \item  The length of the sheaf $\nu_*\inverse\kl/\kq$ at  $p$ is
\[ l_p(\nu_*\inverse\kl/\kq)=\begin{cases}
                              1 &  \text{if $p$ is a degenerate cusp of $X$,}\\
 0  & \text{else.}
                             \end{cases}\]
\item If $p$ is a degenerate cusp, then   the linear form
$\phi_p\colon \ko_{\bar D^\nu, \inverse{(\pi\circ\bar\nu)}(p)}\to \IC$ defined by 
\[f\mapsto \sum_{i=1}^{\mu} (f(r_i)-f(s_i)),\]
 induces an isomorphism
\[\phi_p\colon \left(\nu_*\inverse \kl/\kq\right)_p \isom
 \left(\frac{\pi_*\left(\bar \nu_*\ko_{\bar D^\nu}/\ko_{\bar D}\right)}{\nu_*\ko_{D^\nu}/\ko_D}\right)_p
\overset{\isom}{\longrightarrow} \IC.\]
\end{enumerate}
\end{lem}
\begin{proof}
In our chosen notation, by diagram \eqref{eq: q2} we can identify
\[\left(\nu_*\inverse \kl/\kq\right)_p\isom\left( \frac{\pi_*\left(\bar \nu_*\ko_{\bar D^\nu}/\ko_{\bar D}\right)}{\nu_*\ko_{D^\nu}/\ko_D}\right)_p=
\frac{\bigoplus_i \left(\ko_{r_i}\oplus\ko_{s_i}\right)/\ko_{q_i}}{\left(\bigoplus_i\ko_{p_i}\right)/\ko_p}.\]
The first part follows from a simple dimension count. For the second part we only need to note that the given function is non-zero and descends to the quotient.
\end{proof}
\begin{prop}\label{prop: algorithm q}
 Let $X$ be a connected slc surface with degenerate cusps $p^1, \dots, p^k$. At each $p^i$ choose an isomorphism $\phi_{p^i}$ as in Lemma \ref{lem: L^-1/Q}. Choose a basis $f_1, \dots f_l$ of  $H^0(\nu_*\inverse \kl)=H^0(\inverse \kl)\subset H^0(\ko_{\bar D^\nu})$, viewed as the  $\tau$-anti-invariant functions on $\bar D^\nu$,  and consider the $k\times l$-matrix $M = (\phi_{p^i}(f_j))_{ij}$. Then
\[h^0(\kq) = \dim \ker M.\]
In particular, if the normalisation $\bar X$ is the disjoint union of $m$  surfaces with irregularity $q=0$ then
\[ q(X)=h^0(\kq)-(m-1)=\dim \ker M-m+1.\]
\end{prop}
\begin{proof}
 By Lemma \ref{lem: L^-1/Q} the matrix $M$ represents the map $H^0(\nu_*\inverse\kl)\to H^0(\nu_*\inverse\kl/\kq) )$ in the long exact cohomology sequence associated with the third row of \eqref{eq: q2}, which proves the first part. The second part follows from the second column of \eqref{eq: q1}.
\end{proof}
Note that the basis of $H^0(  \inverse\kl)\subset H^0(\ko_{\bar D^\nu})$ required in the above proposition is easy to choose.  If a component $C\subset \bar D^\nu$ is fixed by $\tau$,  then any $\tau$-anti-invariant function  on $C$ vanishes. If $\tau(C) = C'\neq C$ then the only $\tau$-anti-invariant functions on $C\sqcup C'$ are locally constant of the form $\lambda(1\sqcup -1)$ for $\lambda\in \IC$.

A sample computation in an explicit example is carried out in the proof of Proposition \ref{prop: algebraic invariants}.

\section{Gorenstein stable surfaces with $K^2 = 1$} \label{sec:k2=1}

In this section we prove Theorem A,  stated in the Introduction: it follows by combining Proposition \ref{prop:irregular}, Corollary \ref{cor: 4lines} and Propositions \ref{prop:table},  \ref{prop; topological invariants}, \ref{prop: algebraic invariants}.
Along the way, we give the complete classification of stable Gorenstein surfaces with  normalisation $(\IP^2, \text{4 lines})$.

\subsection{Classical methods to prove regularity}\label{section: classical}

Let $X$ be a stable Gorenstein surface with $K^2_X=1$. 
For $p_g(X)>0$ we look at the canonical curves:
\begin{lem}\label{lem:canonical-curve}
Let $X$ be a Gorenstein stable surface with   $K_X^2=1$  and let $C\
\in |K_X|$ be a canonical curve. 

 Then $C$ is a reduced and irreducible Gorenstein curve with  $p_a(C)=2$, not contained in  the non-normal locus.
\end{lem}
\begin{proof}
By assumption $\ko_X(C) = \omega_X$ is a line bundle, so $C$ is a Gorenstein curve and by adjunction we have   $p_a(C)=2$. Since $K_XC = 1$ and $K_X$ is an ample Cartier divisor, the curve $C$ is reduced and irreducible. Since no component of the non-normal locus is Cartier and  $C$ is reduced,  $C$ cannot be contained in the non-normal locus.
\end{proof}

\begin{prop}\label{prop:irregular}
 Let $X$ be a Gorenstein stable surface with $K_X^2 = 1$. Then  $q(X)>0$ if and only if $\chi(X)= 0$.
\end{prop}
\begin{proof} Clearly a surface with $\chi(X)=0$ is irregular. 
For the other direction, let  $X$ be a Gorenstein stable surface with $K_X^2=1$. 
 We have $\chi(X)\ge 0$ by  \cite[Thm.3.6]{fpr14},  hence it is enough to show that for  $\chi(X)>0$ one has $q(X)=0$. 

So assume by  contradiction that $\chi(X)>0$ and $q(X)>0$.
Since  $h^2(K_X+\eta)=h^0(\eta)=0$ for every $0\ne \eta\in \Pic^0(X)$, it follows that  $h^0(K_X+\eta)\ge \chi(K_X+\eta)=\chi(X)>0$. 
Since $\chi(X)>0$, the linear system $|K_X|$ is nonempty.  Pick $C\in |K_X|$:  since $h^0(\eta)=0$, there is an injection   $H^0(K_X+\eta)\hookrightarrow H^0((K_X+\eta)|_C)$. But $(K_X+\eta)C=1$ and $C$ is irreducible of genus 2 by Lemma \ref{lem:canonical-curve}, so by Riemann-Roch  we have $h^0((K_X+\eta)|_C)= h^0(K_X+\eta)=\chi(X)=1$.  For $\eta\ne 0$, if we denote by $C_{\eta}$ the only curve in $|K_X+\eta|$   then
$C_{\eta}$ intersects $C$ transversely at a point $P_{\eta}$ which is smooth for $C$.

In addition, by the generalized Kodaira vanishing (cf. \cite[Thm. 3.1]{liu-rollenske14}) one has $H^1(\OO_X(-C))=0$, hence $H^1(\OO_X)\to H^1(\OO_C)$ is an injection and  the homomorphism of algebraic groups $\Pic^0(X)\to \Pic^0(C)$ has finite kernel.
 It follows that   the map  ${\eta}\to P_{\eta}$, $\eta\in \Pic^0(C)\setminus \{0\}$,  is not constant.  So  the system $|2K_X|\restr C$ contains the infinitely many divisors $P_{\eta}+P_{-\eta}$. Since $h^0(K_C)=2$, it follows that the  restriction map $H^0( 2K_X)\to H^0(2K_X|_C)=H^0(K_C)$ is surjective. 

By the generalized Kodaira vanishing and Serre duality  we have $H^1(2K_X)=0$ and the  adjunction  sequence  $0\to K_X\to 2K_X\to K_C\to 0$ gives an exact sequence:  
\[H^0(2K_X) \to H^0(K_C)  \to H^1(K_X) \to 0,\]
and therefore $q(X)=h^1(K_X)=0$,  contradicting the assumptions. 
\end{proof}

\begin{cor}\label{cor: 4lines}
Let $X$ be an irregular Gorenstein stable surface with $K^2=1$. Then $X$ is non-normal and the normalisation of $X$ is $(\IP^2, \text{4 general lines})$.
\end{cor}
\begin{proof} By Proposition \ref{prop:irregular} we have $\chi(X)=0$. 
 By \cite[Thm. 3.6]{fpr14} $X$ is non-normal and its normalisation is $\IP^2$ with  conductor a stable quartic.  We have $\chi(D)=\chi(X)-\chi(\bar X)+\chi(\bar D)=-3$, e.g., by \eqref{eq: formula chi}. So in this case the inequality \refenum{ii} of \cite[Lem. 3.5]{fpr14} is an equality, and by  ibidem $\bar D$ has rational components and 6 nodes. It follows that $\bar D$ is the union of four lines in general position.   
\end{proof}

\subsection{Interlude: Gorenstein stable surfaces from four lines in the plane}\label{section: 4 lines}
Motivated by Corollary \ref{cor: 4lines}, we classify here  Gorenstein stable surfaces with normalisation $(\IP^2, \text{4 lines})$. 
We take Koll\'ar's approach (cf. Section \ref{ssec: kollar}) to the classification of stable surfaces, as obtained from an lc pair    $(\bar X, \bar D)$ by gluing $\bar X$ along $\bar D$ via an involution  $\tau$ of the normalization $\bar D^{\nu}$ of $\bar D$. 

We take $\bar X=\pp^2$ and $\bar D=L_1+\cdots +L_4$, where $L_1=\{x_0=0\}$, $L_2=\{x_1=0\}$, $L_3=\{x_2=0\}$  and $L_4=\{x_0+x_1+x_2=0\}$,  and classify all the stable Gorenstein surfaces  that arise  from this lc pair; as a byproduct we obtain the complete classification of the stable Gorenstein surfaces with $K^2_X=1$ and $\chi(X)=0$ (Proposition \ref{prop: M_{0,1}}).

Denote by $P_{(ij)}\in \bar D$ the intersection point of $L_i$ and $L_j$. The normalization of $\bar D$ is 
 $\bar D^\nu = \bigsqcup L_i$:   we denote by  $P_{ij}$  the point  of $L_i\subset \bar D^{\nu}$  that maps to $P_{(ij)}$, so that  each component of $\bar D^{\nu}$ contains three marked points.  Recall  (cf. Addendum to Thm. \ref{thm: triple}) that  an involuton $\tau$ of $\bar D^{\nu}$ gives rise to a Gorenstein stable surface if and only if it induces a fixed point free involution of  the marked points. Since every component of $\bar D^{\nu}$ contains three such points, $\tau$ cannot preserve any of the $L_i$,   so we may assume that it maps    $L_1$ to   $L_2$ and $L_3$ to $L_4$.
Then $\tau$ is uniquely determined by two bijections 
\begin{gather*}
\phi_{12}\colon \{ P_{12}, P_{13}, P_{14}\}\to \{P_{21}, P_{23}, P_{24}\},\\
\phi_{34}\colon\{P_{31}, P_{32}, P_{34}\}\to \{P_{41},P_{42}, P_{43}\}.
\end{gather*}

The set of possible choices of $(\phi_{12}, \phi_{34})$ can be identified with $S_3\times S_3$.
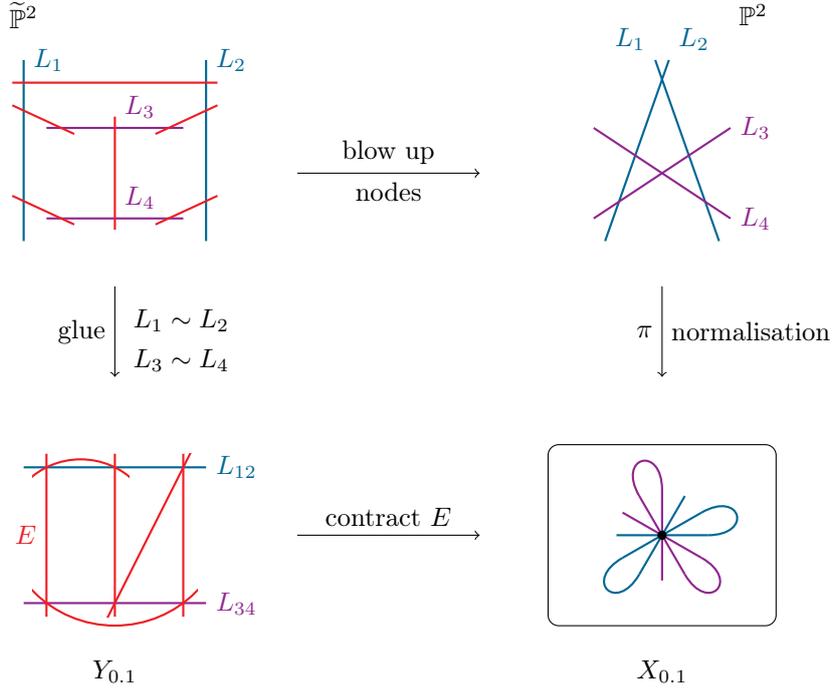
\begin{figure}
\small
 \begin{tikzpicture}
[exceptional/.style = {Red, thick}, 
L12/.style ={thick, MidnightBlue},
L34/.style={thick, Plum},
scale = .6]

\begin{scope}[xshift = 12cm, yshift = 8cm]
\node at (2,2) {$\IP^2$};
 \draw[L12] (-.15, 1) node[above left]{$L_1$} -- (1.25, -3)  (.15, 1)node[above right]{$L_2$} -- (-1.25, -3);
\draw[L34] (-1.5, -2.5) -- (1.5, -.5) node[right]{$L_3$} (-1.5, -.5) -- (1.5, -2.5)node[right]{$L_4$};
\draw[->] (0, -4) to  node[left] {$\pi$} node[right] {\small normalisation} ++(0,-2);
\end{scope}

\begin{scope}[xshift = 0cm, yshift = 8cm]
 \node at (-2,2) {$\tilde \IP^2$};

\begin{scope}
 \draw[L12] (-2,1) node[right]{$L_1$} -- (-2, -3) (2,1)node[right]{$L_2$} -- (2, -3);
\draw[L34] (-1.5, -.5) -- ( 1.5, -.5)  (0,-.5) node [above right] {$L_3$}  (-1.5, -2.5) -- ( 1.5, -2.5) (0,-2.5) node [above right] {$L_4$};
\draw[->] (4, -1.5) to node[above] {\small blow up} node[below] {\small nodes} ++(4,0);
\end{scope}
\draw[exceptional] (-2.25, .5) -- (2.25, .5);
\begin{scope}[exceptional, yshift=.5cm]
\draw (0, -.75) -- (0, -3.25);
\draw (-2.25, -.5)  -- ++(-25:1.5cm);
\draw (-2.25, -2.5)  -- ++(-25:1.5cm);
\draw (2.25, -.5)  -- ++(205:1.5cm);
\draw (2.25, -2.5)  -- ++(205:1.5cm);
\end{scope}
\draw[->] (0, -4) to  node[left] {\small glue} node[right] {\begin{minipage}[c]{1.5cm}
\small\begin{align*} L_1&\sim L_2\\ L_3&\sim L_4\end{align*}\end{minipage}}
 ++(0,-2);
\end{scope}

\begin{scope}
\draw[L12] (-2, 0)-- (2,0) node[right]{$L_{12}$};
\draw[L34] (-2, -3)-- (2,-3) node[right]{$L_{34}$};

\begin{scope}[exceptional]
 \draw (-1.5, .3)to node[left]{$E$} ++(0, -3.6);
 \draw (0, .3)-- ++(0, -3.6);
 \draw (1.5, .3)-- ++(0, -3.6);
\begin{scope}
 \clip  (-1.5, .3) rectangle (2, -3.3);
\draw (-1.5, -6)--(3, 3);
\end{scope}
\begin{scope}
\clip (-1.8, -1) rectangle (1.8, -4);
\node [draw] at (0,-1) [circle through={(-1.5,-3)}] {};
\end{scope}
 \clip (-1.8, .4) rectangle (.3, -1);
\node [draw] at (-.75,-1.5) [circle through={(0,0)}] {};
\end{scope}
\node  at (0, -4.5) {$Y_{0.1}$};
\draw[->] (4, -1.5) to node[above] {\small contract $E$}  ++(4,0);
\end{scope}

\begin{scope}
[xshift = 12cm, yshift =-1.5cm,  looseness=5]
\draw[rounded corners] (-2.5, 2) rectangle (2.5, -2);
\node at (0, -3) {$X_{0.1}$};
 \path (0,0) coordinate(P) ;
\foreach \x in {0,30,60} {\draw[L12] (P) ++(\x:1) -- (\x:-1);}
\foreach \x in {0, 210} {\draw[L12] (P) ++(\x:1) to[out=\x, in =\x+30]  (\x+30:1);}
\foreach \x in {90, 120, 150} {\draw[L34] (P) ++(\x:1) -- (\x:-1);}
\foreach \x in {90, -60} {\draw[L34] (P) ++(\x:1) to[out=\x, in =\x+30]  (\x+30:1);}
\draw[fill=black, radius = 2.5pt] (P)  circle;
\end{scope}
\end{tikzpicture}
\caption{Construction of $X_{0.1}$}\label{fig: X_0.1 construction}
\end{figure}
\begin{exam}
 To visualise what is going on it is helpful to take a log-resolution of $(\bar X, \bar D)$, that is, blow up the intersection points of the four lines and then glue the strict transform of $\bar D$, which is $\bar D^\nu$. On the resulting surface one can  contract the exceptional curves to degenerate cusps.  In technical terms, we  construct the minimal semi-resolution and then pass to the canonical model.

Let us illustrate this procedure in an explicit example: consider the involution given by 
\[\phi_{12}=\begin{pmatrix}
 P_{12} & P_{13} & P_{14}\\
 P_{23} & P_{21} &P_{24}
\end{pmatrix}, \, \phi_{34}=\begin{pmatrix}
 P_{31} & P_{32} & P_{34}\\
 P_{41} & P_{43} &P_{42}
\end{pmatrix}.\]
Glueing the strict transforms of the lines in the blow up of $\IP^2$ gives a semi-smooth surface $Y_{0.1}$; contracting the exceptional curves on $Y_{0.1}$ yields the stable surface $X_{0.1}$. This is illustrated in Figure \ref{fig: X_0.1 construction}.

We see that the non-normal locus $D$ hast two components, each a rational curve with a triple point. At the intersection point of the curves, the surface  $X_{0.1}$ has a degenerate cups $P$, which locally look like the cone over a circle of six independent lines in projective space, that is, locally analytically $(X_{0.1}, P)$ is  isomorphic to the  vertex of the cone in $\IA^6$ given by equations $ \{z_iz_j \mid i-j\neq -1, 0 , 1 \mod 6\}$
Thus blowing up the degenerate cusp $P$ results in  the  semi-resolution $Y_{0.1}$, the exceptional divisor $E$ being a cycle of six smooth rational curves. 

We will study this surface and its cousin $X_{0.2}$ more in detail in the next section.
\end{exam}

\begin{prop}\label{prop:table}
Let $X$ be a Gorenstein stable surface such that the normalization $\bar X$ of $X$ is $\pp^2$ and the double locus $\bar D\subset \bar X$ is the union of four lines.
Then $X$ is isomorphic to one (and only one) of the surfaces corresponding to the involutions listed  in Table \ref{tab: 4lines} on page \pageref{tab: 4lines}.
\end{prop}
\begin{proof}
The curve $\bar D$  is nodal since $(\bar X, \bar D)$  is lc, so it consists of four lines in general position and we may assume without  loss of generality  that $D=L_1+\cdots +L_4$ and that $\tau$ interchanges $L_1$ with $L_2$ and $L_3$ with $L_4$.
Every permutation of $L_1,\dots L_4$ is induced by an element of $\Aut(\IP^2)$, so the automorphism group of $(\bar X,\bar D)$ can be identified with $S_4$; our choice of  which lines should be interchanged by 
the involution $\tau$ reduces the symmetry group to $D_4$, generated by the involutions  $(12)$, $(34)$ and $(13)(24)$. The symmetry group $D_4$ acts on these choices by permutation of the indices;  by Lemma \ref{lem:autXbar} the stabiliser is the automorphism group of the corresponding stable surface. 

In Table \ref{tab: 4lines}  we give a representative for each orbit,  together with some information on the stable surface $X$ constructed from the triple $(\bar X, \bar D, \tau)$. 
Given this data one can prove that the list is complete  and without redundancies in the following way: first check that the stabilisers are correct and then use the orbit stabiliser theorem to show that the given orbits fill up $S_3\times S_3$ if they are disjoint. Apart from case 1.2 and 1.3, where an explicit computation is needed, no   two orbits can be equal because either the stabilisers are not conjugate in $D_4$ or the structure of the resulting degenerate cusps is different.

Let $\mu$ be the number of degenerate cusps of $X$. Then by \eqref{eq: formula chi} and \cite[Lem.~3.5]{fpr14} we have
$\chi(X)=\chi(\bar X)-\chi(\bar D)+\chi(D)= \mu-1$.

The irregularity $q(X)$ vanishes if $\chi(\ko_X)\geq 1$ by Proposition \ref{prop:irregular}.  The fact that $q(X)=1$ if $\chi(X)=0$ will be proved below in Proposition \ref{prop: algebraic invariants}.
\end{proof}

\subsection{Irregular Gorenstein stable surfaces with $K^2 = 1$}\label{section: chi=0}
By Corollary \ref{cor: 4lines} and Proposition \ref{prop:table} there are exactly two irregular Gorenstein stable surfaces with $K^2=1$.  The next result takes a moduli perspective on these surfaces.

\begin{prop}\label{prop: M_{0,1}}
The moduli space  $\overline \gothM_{1,0}^{ (Gor)}$ of Gorenstein stable surfaces with $K_X^2 =1$ and $\chi(\ko_X)=0$  consists of two isolated points corresponding to the surfaces  $X_{0.1}$ and $X_{0.2}$  in Table \ref{tab: 4lines} on page \pageref{tab: 4lines}; each point is  a connected component of the moduli space of stable surfaces. 

Moreover, $\overline \gothM_{1,0}^{ (Gor)}$ coincides with the moduli space of {\em irregular} Gorenstein stable surfaces with $K^2 = 1$. 
\end{prop}
We do not know whether there are any irregular normal stable surfaces with $K^2=1$; of course,  such  surfaces would  not be Gorenstein by Proposition  \ref{prop:irregular}. 
\begin{proof}  The last part is Proposition \ref{prop:irregular}. 
The set-theoretical description of $\overline \gothM_{1,0}^{ (Gor)}$ is just Corollary \ref{cor: 4lines} together with Proposition \ref{prop:table}. It remains to show that the subspace $\overline \gothM_{1,0}^{ (Gor)}$  is a union of connected components of the moduli space of stable surfaces;  this follows from the fact that the Gorenstein condition is open in families \cite[Cor.~3.3.15]{Bruns-Herzog}.
\end{proof}
In the remaining part  of this section we apply the techniques of Section \ref{section: computations} to compute the integral homology, fundamental group, irregularity and the Picard group of the two surfaces $X_{0.1}$ and $X_{0.2}$. 

\begin{prop}\label{prop; topological invariants}
The surfaces $X_{0.1} $ and $X_{0.2}$ have the following topological invariants:
\begin{enumerate}
\item $\pi_1(X_{0.1})\isom \langle A,B \ | \ A^{-1}B^{-1}A^2B^{2}\rangle$ and 
 $\pi_1(X_{0.2})\isom \langle A,B \ | \ AB^{-1}A^2B^{2}\rangle$
 and these two groups are not isomorphic. 
\item 
 $ H_i(X_{0.1}, \IZ) =H_i(X_{0.2}, \IZ) =\begin{cases}
                   \IZ & i = 0,1, 2, 4\\
\IZ^2& i = 3
                  \end{cases}
$
\end{enumerate}
\end{prop}
The fact that the two groups in \refenum{i} are not isomorphic was explained to us by Kai-Uwe Bux.

\begin{proof}
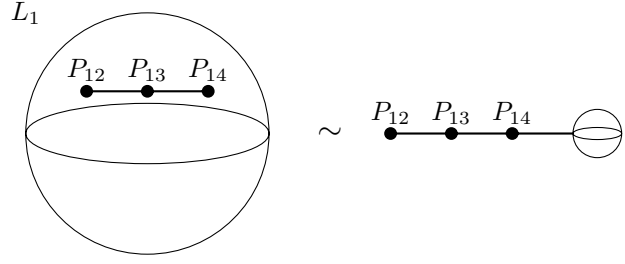
\begin{figure}[t]
\small
 \begin{tikzpicture}
[sphere/.style = {thin},
graph/.style = {thick, fill=black, radius = 2.5pt}, 
scale = .8]
\begin{scope}
\node at (-2,2) {$L_1$};
\draw[sphere] (0,0) circle [radius = 2] circle [x radius=2cm, y radius=5mm];
\draw[graph] (-1,.7) circle node[above] {$P_{12}$}--++ (1,0) circle node[above] {$P_{13}$}--++ (1,0) circle  node[above] {$P_{14}$}; 
 \node at (3,0) {\large$\sim$};
\draw[graph] (4,0) circle node[above] {$P_{12}$}--++ (1,0) circle node[above] {$P_{13}$}--++ (1,0) circle  node[above] {$P_{14}$}--++(1,0); 
\draw[sphere] (7.4,0) circle [radius = .4] circle [x radius=.4cm, y radius=.1cm];
\end{scope}
\end{tikzpicture}
\caption{Homotopically equivalent model of $L_1$ with three marked points}\label{fig: line + points}
\end{figure}
 We continue to use the notation from \eqref{diagr: pushout}. 
In order to make explicit  computations,  we choose a homotopy-equivalent model for $\pi\colon \bar D\to D$: each line $L_i$ is topologically  a sphere with three marked points, the intersection points. Choosing an order for these three points,  this space is homotopy equivalent to the 1-point-union of  an interval  with three marked points and a sphere as in Figure \ref{fig: line + points}. Note that this and the following are \emph{real} pictures. Doing this for all four lines we may choose the order of the points so that the action of $\tau$ on $\bar D^\nu$ is compatible with our model, as in Figures \ref{fig: bar Dnu 0.1} and \ref{fig: bar Dnu 0.2}.
Glueing the four components back together we get a model for $\bar D$, while first taking the quotient by the involution and then glueing gives a model for $D$; this, together with  the map $\pi$,  is  shown in Figure \ref{fig: pi 0.1} for $X_{0.1}$ and   in Figure \ref{fig: pi 0.2} for $X_{0.2}$. 
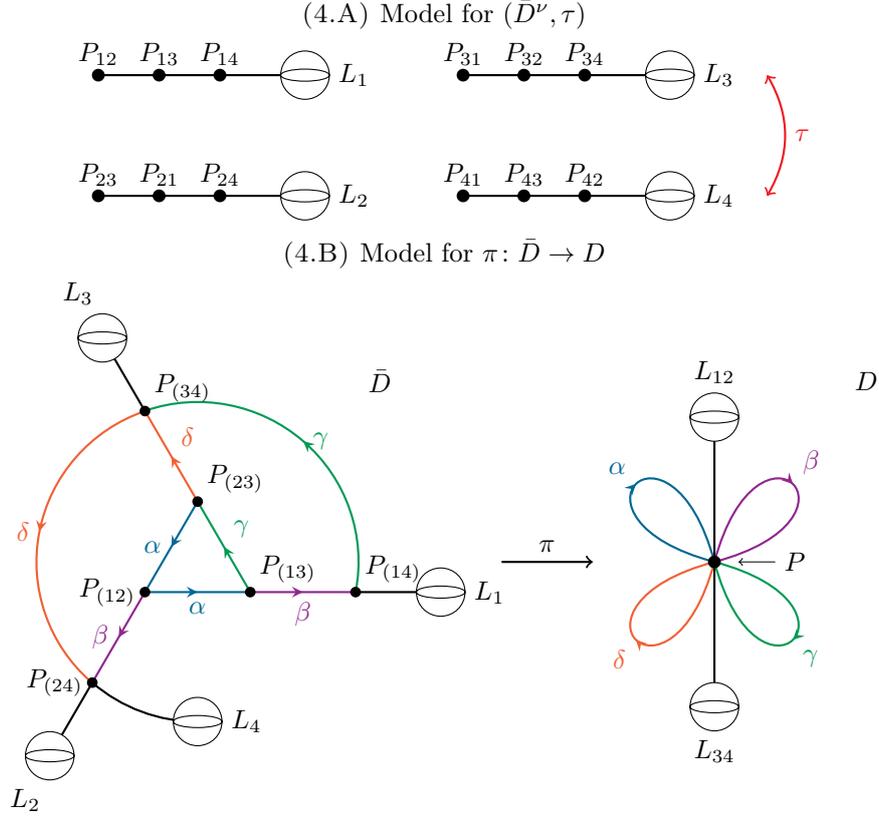
\begin{figure}[t]
\small
\begin{subfigure}{\textwidth}
\centering
\caption{Model for $(\bar D^\nu, \tau)$}\label{fig: bar Dnu 0.1}
 \begin{tikzpicture}
[sphere/.style = {thin},
graph/.style = {thick, fill=black, radius = 2.5pt}, 
scale = .8]
\begin{scope}
\draw[graph] (-2,0) circle node[above] {$P_{12}$}--++ (1,0) circle node[above] {$P_{13}$}--++ (1,0) circle  node[above] {$P_{14}$}--++(1,0); 
\draw[sphere] (1.4,0) circle [radius = .4] circle [x radius=.4cm, y radius=.1cm]++(.4,0) node[right] {$L_1$};

\draw[graph] (-2,-2) circle node[above] {$P_{23}$}--++ (1,0) circle node[above] {$P_{21}$}--++ (1,0) circle  node[above] {$P_{24}$}--++(1,0); 
\draw[sphere] (1.4,-2) circle [radius = .4] circle [x radius=.4cm, y radius=.1cm] ++(.4,0) node[right] {$L_2$};;
 
\draw[graph] (4,0) circle node[above] {$P_{31}$}--++ (1,0) circle node[above] {$P_{32}$}--++ (1,0) circle  node[above] {$P_{34}$}--++(1,0); 
\draw[sphere] (7.4,0) circle [radius = .4] circle [x radius=.4cm, y radius=.1cm]++(.4,0) node[right] {$L_3$};
 
\draw[graph] (4,-2) circle node[above] {$P_{41}$}--++ (1,0) circle node[above] {$P_{43}$}--++ (1,0) circle  node[above] {$P_{42}$}--++(1,0); 
\draw[sphere] (7.4,-2) circle [radius = .4] circle [x radius=.4cm, y radius=.1cm]++(.4,0) node[right] {$L_4$};
\draw[thick, <->, bend left, Red] (9,0) to node[right]{$\tau$} (9,-2);
 \end{scope}
\end{tikzpicture}
        \end{subfigure}

        \begin{subfigure}{\textwidth}
\caption{Model for $\pi\colon \bar D\to D$ }\label{fig: pi 0.1}
\small\centering

\begin{tikzpicture}[sphere/.style = {thin},
graph/.style = {thick},
alpha/.style = {MidnightBlue, thick, postaction={mid arrow}},
beta/.style = {Plum,thick, postaction={mid arrow}},
gamma/.style = {ForestGreen, thick, postaction={mid arrow}},
delta/.style = {RedOrange, thick, postaction={mid arrow}},
scale = .8]

\begin{scope}
\path
(0,0)++(210:1)  coordinate(P12) ++(-120:1.73205080757) coordinate (P24)
(0,0)++(330:1) coordinate(P13) ++(0:1.73205080757) coordinate (P14)
(0,0)++(90:1) coordinate(P23) ++(120:1.73205080757) coordinate (P34);

\draw[sphere] (P14) ++(1.4,0) circle [radius = .4] circle [x radius=.4cm, y radius=.1cm]++(.4,0) node[right] {$L_1$}; \draw[thick] (P14) --++(1,0);

\draw[thick] (P24) --  ++ (240:1) ;
\draw[sphere] (P24)  ++ (240:1.4) circle [radius = .4] circle [x radius=.4cm, y radius=.1cm]++(0,-.4) node[below left] {$L_2$};

\draw[thick] (P34) --  ++ (120:1) ;
\draw[sphere] (P34)  ++ (120:1.4) circle [radius = .4] circle [x radius=.4cm, y radius=.1cm]++(0,0.4) node[above left] {$L_3$};

\draw[thick]  (P24)  arc(229:270: 2.64575131);
  \draw[sphere, fill=white] (0,0)++(270:2.64575131) circle [ radius = .4] circle [x radius=.4cm, y radius=.1cm]++(.4,0) node[ right] {$L_4$};

\draw[alpha] (P12) to node[below]{$\alpha$} (P13);
\draw[beta] (P13) to node[below]{$\beta$} (P14);
\draw[alpha] (P23)to node[left]{$\alpha$}(P12);
\draw[beta] (P12)to node[left]{$\beta$}(P24);
\draw[gamma] (P13)to node[above right]{$\gamma$}(P23);
\draw[delta] (P23)to node[above right]{$\delta$}(P34);
\draw[gamma] (P14)  arc  (-11:49:2.64575131) node[right ]{$\gamma$} arc  (49:109:2.64575131);
\draw[delta] (P34) arc(109:169:2.64575131)node[left]{$\delta$}  arc(169:229:2.64575131);

  \path[fill=black, radius = 2.5pt]
 (P34) circle   node[above right] {$P_{(34)}$}
(P24) circle  node[  left] {$P_{(24)}$}
(P13) circle node[above right] {$P_{(13)}$}
(P12) circle node[ left] {$P_{(12)}$}
(P14) circle node[above right] {$P_{(14)}$}
(P23) circle node[above right] {$P_{(23)}$};
\draw[thick, ->] (5,0) to node[above]{$\pi$} ++(1.5,0);
\node at  (3,3) {$\bar D$};
\end{scope}

\begin{scope}[xshift = 8.5cm, thick, every loop/.style={looseness=40, min distance=80}]
\draw (0,-2) -- (0,2);
\draw[sphere] (0,-2.4) circle [radius = .4] circle [x radius=.4cm, y radius=.1cm]++(0,-.4) node[below] {$L_{34}$};
\draw[sphere] (0,2.4) circle [radius = .4] circle [x radius=.4cm, y radius=.1cm]++(0,.4) node[above] {$L_{12}$};

\path (0,0) coordinate(P) ;\draw[->, very thin] (1,0) node[right] {$P$} to (.4,0);
 \draw[gamma] (0,0) to[in=-75, out =-15,loop]node[below right] {$\gamma$} ();
 \draw[beta] (0,0) to[in=15, out =75,loop] node[above right]{$\beta$} ();
 \draw[alpha] (0,0) to[out=165, in =105,loop] node[above left] {$\alpha$} ();
 \draw[delta] (0,0) to[out=-105, in =-165,loop] node[below left]{$\delta$}();
\draw[fill=black, radius = 2.5pt]   circle;
\node at (2.5,3) {$D$};
\end{scope}
\end{tikzpicture}
        \end{subfigure}
        \caption{Homotopy-equivalent models in case $X_{0.1}$}\label{fig: X0.1}
\end{figure}
\begin{figure}[t]
\small
\begin{subfigure}{\textwidth}
\centering
\caption{Model for $(\bar D^\nu, \tau)$}\label{fig: bar Dnu 0.2}
 \begin{tikzpicture}
[sphere/.style = {thin},
graph/.style = {thick, fill=black, radius = 2.5pt}, 
scale = 1]
\begin{scope}
\draw[graph] (-2,0) circle node[above] {$P_{12}$}--++ (1,0) circle node[above] {$P_{13}$}--++ (1,0) circle  node[above] {$P_{14}$}--++(1,0); 
\draw[sphere] (1.4,0) circle [radius = .4] circle [x radius=.4cm, y radius=.1cm]++(.4,0) node[right] {$L_1$};

\draw[graph] (-2,-2) circle node[above] {$P_{23}$}--++ (1,0) circle node[above] {$P_{24}$}--++ (1,0) circle  node[above] {$P_{21}$}--++(1,0); 
\draw[sphere] (1.4,-2) circle [radius = .4] circle [x radius=.4cm, y radius=.1cm] ++(.4,0) node[right] {$L_2$};;
 
\draw[graph] (4,0) circle node[above] {$P_{31}$}--++ (1,0) circle node[above] {$P_{32}$}--++ (1,0) circle  node[above] {$P_{34}$}--++(1,0); 
\draw[sphere] (7.4,0) circle [radius = .4] circle [x radius=.4cm, y radius=.1cm]++(.4,0) node[right] {$L_3$};
 
\draw[graph] (4,-2) circle node[above] {$P_{41}$}--++ (1,0) circle node[above] {$P_{43}$}--++ (1,0) circle  node[above] {$P_{42}$}--++(1,0); 
\draw[sphere] (7.4,-2) circle [radius = .4] circle [x radius=.4cm, y radius=.1cm]++(.4,0) node[right] {$L_4$};
\draw[thick, <->, bend left, Red] (9,0) to node[right]{$\tau$} (9,-2);
 \end{scope}
\end{tikzpicture}
 \end{subfigure}

 \begin{subfigure}{\textwidth}
\caption{Model for $\pi\colon \bar D\to D$ }\label{fig: pi 0.2}
\centering
\begin{tikzpicture}[sphere/.style = {thin},
graph/.style = {thick},
alpha/.style = {MidnightBlue, thick, postaction={mid arrow}},
beta/.style = {Plum,thick, postaction={mid arrow}},
gamma/.style = {ForestGreen, thick, postaction={mid arrow}},
delta/.style = {RedOrange, thick, postaction={mid arrow}},
scale = .8]

\begin{scope}
\path
(-4,0)  coordinate(P12)
(-2,0) coordinate(P24)
(0,0) coordinate(P23)
(0,2) coordinate(P34)
(0,-2) coordinate(P13)
(2,-2) coordinate(P14);

\draw[sphere] (P14) ++(2,0) circle [radius = .4] circle [x radius=.4cm, y radius=.1cm]++(.4,0) node[right] {$L_1$};
\draw[thick] (P14) --++(1.6,0);

\draw[sphere] (P12) ++(-1.4,0) circle [radius = .4] circle [x radius=.4cm, y radius=.1cm]++(-.4,0) node[left] {$L_2$};
\draw[thick] (P12) --++(-1,0);

\draw[sphere] (P24) ++(0,-1) circle [radius = .4] circle [x radius=.4cm, y radius=.1cm]++(.4,0) node[right] {$L_4$};
\draw[thick] (P24) --++(0,-.6);

\draw[sphere] (P34) ++(0,1.4) circle [radius = .4] circle [x radius=.4cm, y radius=.1cm]++(0,.4) node[above] {$L_3$};
\draw[thick] (P34) --++(0,1);

\draw[beta] (P13) to node[below]{$\beta$} (P14);
\draw[alpha] (P12)  arc(-180:-90:2cm) node[below]{$\alpha$} to ++(2,0);
\draw[alpha] (P23)to node[below]{$\alpha$}(P24);
\draw[beta] (P24)to node[below]{$\beta$}(P12);
\draw[gamma] (P13)to node[right]{$\gamma$}(P23);
\draw[delta] (P23)to node[right]{$\delta$}(P34);
\draw[gamma] (P14) --(2,0)  node[right]{$\gamma$} arc(0:90:2cm) ;
\draw[delta] (P34) arc (90:135:2cm) node[above left]{$\delta$} arc (135:180:2cm);

  \path[fill=black, radius = 2.5pt]
 (P34) circle   node[above right] {$P_{(34)}$}
(P24) circle  node[above right] {$P_{(24)}$}
(P13) circle node[above right] {$P_{(13)}$}
(P12) circle node[above] {$P_{(12)}$}
(P14) circle node[above right] {$P_{(14)}$}
(P23) circle node[above right] {$P_{(23)}$};
\draw[thick, ->] (4.25,0) to node[above]{$\pi$} ++(1.5,0);
\node at  (3,3) {$\bar D$};
\end{scope}

\begin{scope}[xshift = 7.5cm, thick, every loop/.style={looseness=40, min distance=80}]
\draw (0,-2) -- (0,2);
\draw[sphere] (0,-2.4) circle [radius = .4] circle [x radius=.4cm, y radius=.1cm]++(0,-.4) node[below] {$L_{34}$};
\draw[sphere] (0,2.4) circle [radius = .4] circle [x radius=.4cm, y radius=.1cm]++(0,.4) node[above] {$L_{12}$};

\path (0,0) coordinate(P) ;\draw[->, very thin] (1,0) node[right] {$P$} to (.4,0);
 \draw[gamma] (0,0) to[in=-75, out =-15,loop]node[below right] {$\gamma$} ();
 \draw[beta] (0,0) to[in=15, out =75,loop] node[above right]{$\beta$} ();
 \draw[alpha] (0,0) to[out=165, in =105,loop] node[above left] {$\alpha$} ();
 \draw[delta] (0,0) to[out=-105, in =-165,loop] node[below left]{$\delta$}();
\draw[fill=black, radius = 2.5pt]   circle;
\node at (2.5,3) {$D$};
\end{scope}
\end{tikzpicture}
\end{subfigure}
\caption{Homotopy-equivalent models in case $X_{0.2}$}\label{fig: X0.2}
\end{figure}
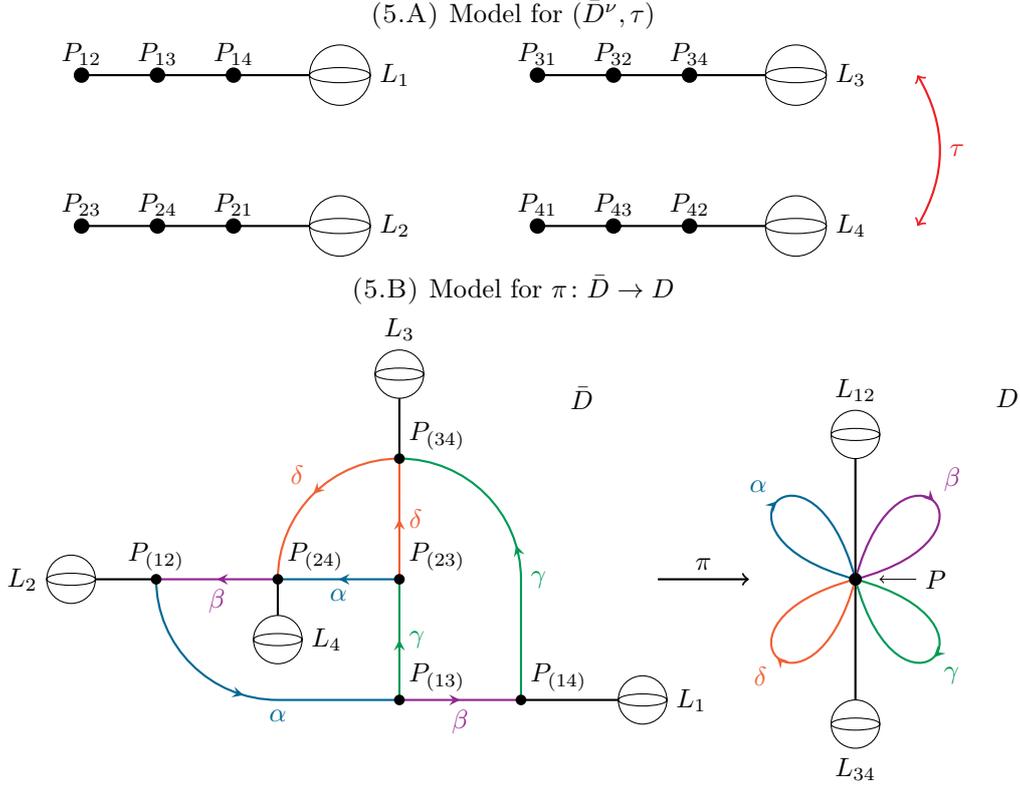

\refenum{i}
 By Corollary \ref{cor: mayer vietoris and fundamental group as amalgamated product} we have 
\[ \pi_1(X_{0.1})= \pi_1( D)\star_{\pi_1(\bar D)}\pi_1(\IP^2)
\isom \frac{\pi_1( D)}{\langle \pi_1(\bar D)\rangle}\] 
We choose as base points $P\in D$ and $P_{(23)}\in \bar D$ (cf. Figures \ref{fig: pi 0.1} and \ref{fig: pi 0.2}).
We read off  our model that $\pi_1(D, P)$ is the free group generated by the loops $\alpha, \beta, \gamma, \delta$ depicted in Figure  \ref{fig: pi 0.1} in the case of $X_{0.1}$ and in Figure  \ref{fig: pi 0.2} in the case of $X_{0.2}$.

If $X=X_{0.1}$,  then  the image of $\pi_1(\bar D, P_{(23)})$ is generated by 
\begin{equation}\label{eq: relation X_0.1}
 \alpha^2\gamma ,  \delta\inverse\gamma \inverse\beta  \gamma, \alpha\beta\delta^{-2} .
 \end{equation}
Solving for $\gamma=\alpha^{-2}$ and $\delta=\gamma^{-1}\beta\gamma$, one sees that the quotient group is generated by (the classes of ) $\alpha$ and $\beta$ with the relation: 
\[\alpha \beta \alpha^2\beta^{-2}\alpha^{-2}.\]
Denoting by $A$ the class of $\alpha$ and by $B$ the class of $\alpha^{-2}\beta^{-1}\alpha^2$ one gets the presentation given in the statement. 

For $X_{0.2}$ the image of $\pi_1(\bar D, P_{(23)})$ is generated by:
\begin{equation}\label{eq: relation X_0.2}
\delta^2\alpha^{-1}, \alpha\beta\alpha\gamma, \delta \gamma^{-1}\beta^{-1}\gamma 
\end{equation}
Solving for $\alpha=\delta^{2}$ and $\beta=\gamma\delta \gamma^{-1}$, one sees that the quotient group is generated by (the classes of ) $\gamma$ and $\delta$ with the relation: 
\[
\delta^2\gamma\delta\gamma^{-1}\delta^2\gamma.
\]
Denoting by $A$ the class of $\delta$ and by $B$ the class of $\delta^2\gamma$ the relation becomes $BAB^{-1}A^2B$ and one gets the presentation given in the statement after conjugation with $B$.

To prove that the two groups are not isomorphic, we show that $\pi_1(X_{0.1})$ admits a surjective isomorphism onto the group $A_4$, while  $\pi_1(X_{0.2})$ does not.

The homomorphism $\psi\colon \pi_1(X_{0.1})\to A_4$ is defined by  $A\mapsto (234)$ and $B\mapsto (123)$: it is easy to check that $\psi(A^{-1}B^{-1}A^2B^{2})=1$.

On the other hand, assume for contradiction that  $\psi\colon \pi_1(X_{0.2})\to A_4$ is a surjective homomorphism and set $a=\psi(A)$ and  $b=\psi(B)$; the equality:
\begin{equation}\label{eq: rel}
ab^{-1}a^2b^2=1,
\end{equation}
 holds in $A_4$. Note that since $\psi$ is surjective at least one betweeen $a$ and $b$ has order 3. If $a$ and $b$ both have order $3$, \eqref{eq: rel} gives $a^{-1}ba=b^{-1}$, and therefore the subgroup  $<b>$ is normal in $A_4$, a contradiction. 
 The remaining cases can be excluded by similar (easier) arguments. 

\refenum{ii} Let $X=X_{0.j}$ for $j= 1, 2$. We apply  the Mayer-Vietoris sequence from Corollary \ref{cor: mayer vietoris and fundamental group as amalgamated product} and get $\pi_*\colon H_4(\IP^2, \IZ)\isom H_4(X, \IZ)$ and  a long exact sequence
\[
 \begin{tikzcd}[ampersand replacement=\&]
{}\&\&0\rar\&H_3(X, \IZ)\rar\&{}\\
 {}\rar\&\underbrace{H_2(\bar D, \IZ)}_{=\IZ\langle [L_{1}], \dots, [L_{4}]\rangle} \rar{N}
 \& \underbrace{H_2(D, \IZ)}_{=\IZ\langle[ L_{12}],[ L_{34}]\rangle} \oplus\underbrace{H_2(\IP^2, \IZ)}_{=\IZ\cdot [L_i]}\rar \& H_2(X, \IZ)
\rar \& {}\\
{}\rar \&H_1(\bar D, \IZ) \rar{M_j} \& H_1(D, \IZ) \oplus\underbrace{H_1(\IP^2, \IZ)}_{=0}\rar \& H_1(X, \IZ)\rar \& 0,
 \end{tikzcd} 
\]
where 
\[N={\begin{pmatrix} 1& 1& 0& 0\\ 0&0&1&1\\1&1&1&1\end{pmatrix}}\]
in the given bases.

We computed the image of $\pi_*$ on fundamental groups in \eqref{eq: relation X_0.1} and \eqref{eq: relation X_0.2} and we get, by abelianisation, matrix representations 
\[
 M_1 = \begin{pmatrix} 2&0&1\\0&-1&1\\1&0&0\\0&1&-2\end{pmatrix} \text{ and } M_2 = \begin{pmatrix} -1&2&0\\0&1&-1\\0&1&0\\2&0&1\end{pmatrix}.
\]
In both cases the map is injective with cokernel isomorphic to $\IZ$. Thus the homology is as claimed.  
\end{proof}

\begin{prop}
\label {prop: algebraic invariants}
Let $X$ be a Gorenstein stable surface with $K^2=1$ and $\chi(X)=0$. Then $q(X) = 1$,   and the exact sequence for $\Pic(X)$ is:
\[0\to \Pic^0(X)\cong \IC^*\to \Pic(X)\to H^2(X,\IZ)\cong \IZ\to 0.\]
\end{prop}
\begin{proof}
By Proposition \ref{prop: M_{0,1}}, $X$ is one of  the surfaces $X_{0.1}$ or  $X_{0.2}$ of  Table \ref{tab: 4lines}. 
In both cases, to compute the irregularity $q$  we follow the method of Proposition \ref{prop: algorithm q}. Let us first consider $X_{0.1}$. We need to establish notation as in Section \ref{section: irregularity} for the points of  $\bar D^\nu$ that map  to the unique degenerate cusp $P\in X_{0.1}$. Let
\begin{align*}
 s_1=P_{12},&& r_2=\tau(s_1) = P_{23},\\
 s_2 = P_{32},&& r_3=\tau(s_2) = P_{43},\\
 s_3 = P_{34},&& r_4=\tau(s_3) = P_{42},\\
 s_4 = P_{24},&& r_5=\tau(s_4) = P_{14},\\
 s_5 = P_{41},&& r_6=\tau(s_5) = P_{31},\\
 s_6 = P_{13},&& r_1=\tau(s_6) = P_{21},\\
\end{align*}
and define a basis $f_1, f_2$ of $\tau$-anti-invariant functions on $\bar D^\nu$ by
\begin{gather*}
 f_1(P_{1\ast}) = 1 = -f_1(P_{2\ast}), \quad f_1(P_{3\ast}) = f_1(P_{4\ast})=0,\\
 f_2(P_{3\ast}) = 1 = -f_2(P_{4\ast}), \quad  f_2(P_{2\ast}) = f_2(P_{2\ast})=0.
\end{gather*}
Then the matrix $M$ from Proposition \ref{prop: algorithm q} is $M = (-2, -2)$ and thus $q(X)=1$ and $p_g(X)=\chi(X)-1+q(X)=0$.

We skip the analogous calculation for the second case.

 Since $h^2(\OO_X)=p_g(X)=0$ for  both $X=X_{0.1}$ and $X=X_{0.2}$,  the exponential sequence gives:
\[
 0\to H^1(X, \IZ)\to H^1(X, \ko_X)\to H^1(X, \ko_X^*) \to H^2(X,\IZ)\to 0, 
 \]
 The statement now follows since $h^1(\OO_X)=1$ and  $H^i(X, \IZ) \cong \IZ$  for $i=1,2$  by Proposition \ref{prop; topological invariants}.
\end{proof}
\begin{rem}
 One could also deduce $q(X)=1$ in a less elementary way: an slc surface $X$ is semi-normal and thus by \cite[Thm.~4.1.7]{alexeev02} its $\Pic^0(X)$ has unipotent rank zero. In other words, it is an extension of multiplicative groups and an abelian variety and thus by the exponential sequence $q(X)\le b_1(X)$. For the two surfaces at hand we know $q(X)>0$ and $b_1(X)=1$, thus $q(X)=1$.
\end{rem}

\begin{landscape}
\begin{table}[htb!]\caption{Surfaces from four lines in the plane\\ (Notation as in Section \ref{section: 4 lines})}\label{tab: 4lines}

\begin{center}
 \begin{tabular}{c cc ccc cc}
 \toprule
surface& $\chi(\ko_X)$ &$\phi_{12}$ & $\phi_{34}$& degenerate cusps   & $q(X) $& $\Aut(X)$ \\
\midrule
$X_{3.1}$&
3&
$\begin{pmatrix}
 P_{12} & P_{13} & P_{14}\\
 P_{21} & P_{24} &P_{23}
\end{pmatrix}$
&
$\begin{pmatrix}
 P_{31} & P_{32} & P_{34}\\
 P_{42} & P_{41} &P_{43}
\end{pmatrix}$
&
$\{P_{(12)}\},\{P_{(34)}\}, \{P_{(13)}, P_{(24)}\}, \{P_{(23)}, P_{(14)}\}$
&
0&
$D_4$

\\
\midrule
$X_{2.1}$&
2&
$\begin{pmatrix}
 P_{12} & P_{13} & P_{14}\\
 P_{21} & P_{23} &P_{24}
\end{pmatrix}$
&
$\begin{pmatrix}
 P_{31} & P_{32} & P_{34}\\
 P_{41} & P_{42} &P_{43}
\end{pmatrix}$
&
$\{P_{(12)}\},\{P_{(34)}\}, \{P_{(13)}, P_{(14)}, P_{(23)}, P_{(24)}\}$
& 0&
$D_4$

\\
$X_{2.2}$&
2&
$\begin{pmatrix}
 P_{12} & P_{13} & P_{14}\\
 P_{21} & P_{23} &P_{24}
\end{pmatrix}$
&
$\begin{pmatrix}
 P_{31} & P_{32} & P_{34}\\
 P_{42} & P_{41} &P_{43}
\end{pmatrix}$
&
$\{P_{(12)}\},\{P_{(34)}\}, \{P_{(13)}, P_{(14)}, P_{(23)}, P_{(24)}\}$
&0&
$\langle (12), (34)\rangle$

\\
$X_{2.3}$&
2&
$\begin{pmatrix}
 P_{12} & P_{13} & P_{14}\\
 P_{23} & P_{24} &P_{21}
\end{pmatrix}$
&
$\begin{pmatrix}
 P_{31} & P_{32} & P_{34}\\
 P_{42} & P_{41} &P_{43}
\end{pmatrix}$
&
$\{P_{(12)}, P_{(23)}, P_{(14)}\}, \{P_{(13)}, P_{(24)}\}, \{P_{(34)}\}$
&0&
$\langle(12)(34)\rangle$

\\
\midrule
$X_{1.1}$&
1&
$\begin{pmatrix}
 P_{12} & P_{13} & P_{14}\\
 P_{21} & P_{23} &P_{24}
\end{pmatrix}$
&
$\begin{pmatrix}
 P_{31} & P_{32} & P_{34}\\
 P_{41} & P_{43} &P_{42}
\end{pmatrix}$
&
$\{P_{(12)}\},\{P_{(34)}, P_{(13)}, P_{(14)}, P_{(23)}, P_{(24)}\}$
&0&
$\langle(34)\rangle$
\\
$X_{1.2}$
&
1
&
$\begin{pmatrix}
 P_{12} & P_{13} & P_{14}\\
 P_{21} & P_{23} &P_{24}
\end{pmatrix}$
&
$\begin{pmatrix}
 P_{31} & P_{32} & P_{34}\\
 P_{42} & P_{43} &P_{41}
\end{pmatrix}$
&
$\{P_{(12)}\},\{P_{(34)}, P_{(13)}, P_{(14)}, P_{(23)}, P_{(24)}\}$
&0&
$\langle(12)(34)\rangle$

\\
$X_{1.3}$&
1&
$\begin{pmatrix}
 P_{12} & P_{13} & P_{14}\\
 P_{21} & P_{24} &P_{23}
\end{pmatrix}$
&
$\begin{pmatrix}
 P_{31} & P_{32} & P_{34}\\
 P_{41} & P_{43} &P_{42}
\end{pmatrix}$
&
$\{P_{(12)}\},\{P_{(34)}, P_{(13)}, P_{(14)}, P_{(23)}, P_{(24)}\}$
&0&
$\langle(12)(34)\rangle$

\\
$X_{1.4}$&
1&
$\begin{pmatrix}
 P_{12} & P_{13} & P_{14}\\
 P_{23} & P_{24} &P_{21}
\end{pmatrix}$
&
$\begin{pmatrix}
 P_{31} & P_{32} & P_{34}\\
 P_{42} & P_{43} &P_{41}
\end{pmatrix}$
&
$\{P_{(12)}, P_{(34)},  P_{(14)}, P_{(23)}\},\{P_{(13)}, P_{(24)}\}$
&0&
$\langle(13)(24),(14)(23)\rangle$

\\
$X_{1.5}$&
1&
$\begin{pmatrix}
 P_{12} & P_{13} & P_{14}\\
 P_{23} & P_{24} &P_{21}
\end{pmatrix}$
&
$\begin{pmatrix}
 P_{31} & P_{32} & P_{34}\\
 P_{43} & P_{41} &P_{42}
\end{pmatrix}$
&
$\{P_{(12)}, P_{(23)}, P_{(14)}\}, \{P_{(13)}, P_{(24)}, P_{(34)}\}$
&0&
$\langle(12)(13)(24)\rangle$

\\
\midrule
$X_{0.1}$&
0
&
$\begin{pmatrix}
 P_{12} & P_{13} & P_{14}\\
 P_{23} & P_{21} &P_{24}
\end{pmatrix}$
&
$\begin{pmatrix}
 P_{31} & P_{32} & P_{34}\\
 P_{41} & P_{43} &P_{42}
\end{pmatrix}$
&
$\{P_{(12)}, P_{(34)}, P_{(13)}, P_{(14)}, P_{(23)}, P_{(24)}\}$
&
1&
$\langle(14)(23)\rangle$
\\
$X_{0.2}$&
0&
$\begin{pmatrix}
 P_{12} & P_{13} & P_{14}\\
 P_{23} & P_{24} &P_{21}
\end{pmatrix}$
&
$\begin{pmatrix}
 P_{31} & P_{32} & P_{34}\\
 P_{41} & P_{43} &P_{42}
\end{pmatrix}$
&
$\{P_{(12)}, P_{(34)}, P_{(13)}, P_{(14)}, P_{(23)}, P_{(24)}\}$
&1
&
$\{0\}$
\\

\bottomrule
\end{tabular}
\end{center}
\end{table}
\end{landscape}

%
%

\end{document}